\documentclass{amsart}
\usepackage{amsopn,amstext,amsbsy,amsmath,amsthm,amsfonts,amssymb}
\usepackage{geometry}
\usepackage[latin9]{inputenc}
\theoremstyle{plain}
\newtheorem{theorem}                 {\bf Theorem}        [section]
\newtheorem{proposition}  [theorem]  {\bf Proposition}

\newtheorem{lemma}        [theorem]  {\bf Lemma}

\theoremstyle{definition}
\newtheorem{example}      [theorem]  {\bf Example}
\newtheorem{definition}   [theorem]  {\bf Definition}

\allowdisplaybreaks

\numberwithin{equation}{section}

\begin{document}

\def\nab#1#2#3{\nabla^{\hbox{$\scriptstyle{#1}$}}_{\hbox{$\scriptstyle{#2}$}}{\hbox{$#3$}}}
\def\nabl#1#2{\nabla_{\hbox{$\scriptstyle{#1}$}}{\hbox{$#2$}}}
\def\tnabl#1#2{\hat{\nabla}_{\hbox{$\scriptstyle{#1}$}}{\hbox{$#2$}}}
\def\nnab#1{\nabla_{\hbox{$\scriptstyle{#1}$}}}
\def\tnab#1#2{\widetilde{\nabla}_{\hbox{$\scriptstyle{#1}$}}\hbox{$#2$}}
\def\rn{\mathbb{R}}
\def\cn{\mathbb{C}}
\def\zn{\mathbb{Z}}
\def\rk{\mathbb{K}}
\def\ci{\mathcal{I}}
\def\cL{\mathcal{L}}
\def\v{\mathfrak{v}}
\def\b{\mathfrak{b}}
\def\fr{\mathfrak{r}}
\def\z{\mathfrak{z}}
\def\g{\mathfrak{g}}
\def\gl{\mathfrak{gl}}
\def\sl{\mathfrak{sl}}
\def\k{\mathfrak{k}}
\def\p{\mathfrak{p}}
\def\h{\mathfrak{h}}
\def\s{\mathfrak{s}}
\def\so{\mathfrak{so}}
\def\n{\mathfrak{n}}
\def\m{\mathfrak{m}}
\def\a{\mathfrak{a}}
\def\V{\mathcal{V}}
\def\H{\mathcal{H}}
\def\L{\mathcal{L}}
\def\B{\mathfrak{B}}
\def\Hi{\mathfrak{H}}
\def\U{\mathfrak{U}}
\def\F{\mathcal{F}}
\def\tr{\operatorname{trace}}
\def\sp{\operatorname{span}}
\def\grad{\operatorname{grad}}
\def\div{\operatorname{div}}
\def\ad{\operatorname{ad}}
\def\Ad{\operatorname{Ad}}
\def\Aut{\operatorname{Aut}}
\def\Ric{\operatorname{Ric}}
\def\dim{\operatorname{dim}}
\def\End{\operatorname{End}}
\def\pd#1{\frac{\partial}{\partial #1}}
\def\dd#1{\frac{\operatorname{d}}{\operatorname{d}#1}}
\def\dop#1{\operatorname{d}#1}
\def\SPE#1#2{\langle #1,#2\rangle}
\def\pro{\textsc{Proof}}

\title{Curvature conditions for complex-valued\\
harmonic morphisms}

\date{}

\author{Jonas Nordstr\" om}



\address
{Department of Mathematics, Faculty of Science, Lund University,
Box 118, S-221 00 Lund, Sweden}
\email{Jonas.Nordstrom@math.lu.se}

\begin{abstract}
We study the curvature of a manifold on which there can be defined a complex-valued
submersive harmonic morphism with either, totally geodesic fibers or that is holomorphic with
respect to a complex structure which is compatible with the second fundamental form.

We also give a necessary curvature condition for the existence of complex-valued harmonic
morphisms with totally geodesic fibers on Einstein manifolds.
\end{abstract}

\maketitle

\section{Introduction}

A harmonic morphism is a map between two Riemannian manifolds
that pulls back local harmonic functions to local harmonic functions.
The simplest examples of harmonic morphisms are constant maps,
real-valued harmonic functions and isometries.
A characterization of harmonic morphisms was given by Fuglede and Ishihara,
they showed in \cite{Fuglede} and \cite{Ishihara}, respectively, that harmonic
morphisms are exactly the horizontally weakly conformal harmonic maps. 
If we restrict our attention to the maps where the codomain is a surface then
the harmonic morphisms are the horizontally weakly conformal maps with minimal fibers
at regular points. 

Between two surfaces the harmonic morphisms are exactly the weakly conformal maps. Since the composition
of two harmonic morphisms is again a harmonic morphism, we get that,
locally any harmonic morphism to a surface can be turned into a harmonic morphism to the
complex plane by composing with a weakly conformal map.

Local existence of harmonic morphisms can be characterized in terms of foliations.
If the codomain is a surface then the existence of a local harmonic morphism is equivalent to
the existence of a local conformal foliation with minimal fibers at regular points, see \cite{Wood86} by Wood.

Baird and Wood found a necessary condition, see \cite{BaiWoo} Corollary 4.4, on the curvature
for local existence of complex-valued harmonic morphisms on three-manifolds.
In this case the fibers are geodesics and there is an orthonormal basis
$\{X,Y\}$ for the horizontal space such that the \textbf{Ricci curvature condition}
\[\Ric(X,X)=\Ric(Y,Y)\textrm{ and }\Ric(X,Y)=0,\]
is satisfied. In three dimensions this is equivalent to
\[\SPE{R(X,U)U}{X}=\SPE{R(Y,U)U}{Y}\textrm{ and }\SPE{R(X,U)U}{Y}=0\]
for any vertical unit vector $U$, which in turn is equivalent to the fact that the
sectional curvature $K(X_{\theta}\wedge U)$ is
independent of $\theta$ where $X_{\theta}=\cos(\theta)X+\sin(\theta)Y$.

We show in this paper that the last condition is true for any complex-valued submersive harmonic morphism with
totally geodesic fibers.
\begin{theorem}\label{Jon-Curv}
Let $(M,g)$ and $(N^{2},h)$ be a Riemannian manifolds, let $\phi:(M,g)\to (N^{2},h)$ be a submersive
harmonic morphism with totally geodesic fibers and $p\in M$.
Given any $U,V\in\V_{p}=\ker(\dop\phi)$ and any orthonormal basis $\{X,Y\}$ for $\H_{p}=\V_{p}^{\bot}$,
set $X_{\theta}=\cos(\theta)X+\sin(\theta)Y$. Then
\[\SPE{R(X_{\theta}\wedge U)}{X_{\theta}\wedge V}\]
is independent of $\theta$.
\end{theorem}
In four dimensions or more this is stronger than the Ricci curvature condition.
Note that Example 6.1 and 6.2 of \cite{Gud-Sven-2013} by Gudmundsson and Svensson
do not have totally geodesic fibers. So they are counterexamples to the Ricci curvature condition
only in the case of minimal but not totally geodesic fibers.
We present these two examples in Example \ref{GudSvenEx1} and \ref{GudSvenEx2}. 

If we assume that the domain $(M,g)$ is an Einstein manifold, then the curvature operator,
in a suitably chosen basis,
splits into two blocks and we find that there are at least $\dim(M)-2$
double eigenvalues for the curvature operator.
We use this to give an example of a five dimensional homogeneous Einstein
manifold that does not have any submersive harmonic morphism with totally geodesic
fibers.

Harmonic morphisms with totally geodesic fibers have been studied in different ways before. Baird and Wood,
Section 6.8 \cite{BW-book}, classify them in the constant curvature case. Later Pantilie generalized
to the case where the domain is conformaly equivalent to constant curvature, \cite{Pantilie08}.
Mustafa \cite{Mustafa} gave a Bochner type curvature formula and applied it to foliations
with large codimension.

We end this paper by showing that the Ricci curvature condition is satisfied by harmonic morphisms
that are holomorphic with respect to a complex structure and where the second fundamental
form is compatible with the complex structure. The result is similar
to Proposition 6.3 from \cite{LouPan}, where Loubeau and Pantilie describe twistorial
harmonic morphisms, but only in $4$ dimensions.

\section{The curvature condition}

Let $(M,g)$ and $(N,h)$ be Riemannian manifolds and let $\phi:(M,g)\to(N,h)$
be a smooth submersion. Denote the vertical distribution associated with $\phi$ by $\V=\ker(\dop\phi)$
and the horizontal distribution by $\H=\V^{\bot}$.
For two vector fields $E,F$ on $M$ define the tensors $A$ and $B$, introduced in \cite{ONeill}, by
\[A_{E}F=\V(\nabla_{\H E}\H F)\textrm{ and }B_{E}F=\H(\nabla_{\V E}\V F).\]
$B$ is called the second fundamental form and the fibers of $\phi$ are said to be totally geodesic if $B=0$.
The dual $A_{X}^{*}$ of $A_{X}$ satisfies $A^{*}_{X}F=-\H(\nabla_{X}\V F)$ for $X\in \H$ and the dual $B_{U}^{*}$
of $B_{U}$ satisfies $B^{*}_{U}F=-\V(\nabla_{U}\H F)$ for $U\in\V$.

Gudmundsson calculated the curvature for a horizontally conformal submersion in \cite{Gud-thesis}, we state the
results from Proposition 2.1.2, and Theorem 2.2.3 (2) and (3) below.
\begin{proposition}\label{Gud-curv}
Let $(M,g)$ and $(N,h)$ be Riemannian manifolds and let $\phi:(M,g)\to(N,h)$ be a
horizontally conformal submersion with dilation $\lambda:M\to(0,\infty)$.
Let $U,V,W$ be vertical vectors and $X,Y$ be horizontal vectors, then
\begin{align*}
(i)&\,A_{X}Y=\frac{1}{2}\V([X,Y])+\SPE{X}{Y}\V(\grad\ln\lambda)\\
(ii)&\,\SPE{R(U\wedge V)}{W\wedge X}=\SPE{(\nabla_{U}B)_{V}W}{X}-\SPE{(\nabla_{V}B)_{U}W}{X}\\
(iii)&\,\SPE{R(U\wedge X}{Y\wedge V}=\SPE{(\nabl{U}{A})_{X}Y}{V}+\SPE{A^{*}_{X}U}{A^{*}_{Y}V}+\SPE{(\nabla_{X}B^{*})_{U}Y}{V}\\
&-\SPE{B^{*}_{V}Y}{B^{*}_{U}X}-2 V(\ln \lambda)\SPE{A_{X}Y}{U}.
\end{align*}
\end{proposition}
It is well-known that $A_{X}Y+A_{Y}X=2\SPE{X}{Y}\V(\grad\ln\lambda)$ and $A_{X}Y-A_{Y}X=\V([X,Y])$.
Suppose that $(N,h)$ is a surface and $\{X,Y\}$ is an orthonormal basis for $\H$ and $U\in \V$ then
\[A^{*}_{X}U=\SPE{A^{*}_{X}U}{X}X+\SPE{A^{*}_{X}U}{Y}Y=\SPE{U}{A_{X}X}X+\SPE{U}{A_{X}Y}Y.\]
From this we see that $\SPE{A^{*}_{X}U}{A^{*}_{X}U}$ does not depend on the direction of $X$. First
\begin{align*}
\SPE{A^{*}_{X}U}{A^{*}_{X}U}=&\SPE{U}{A_{X}X}^{2}+\SPE{U}{A_{X}Y}^{2}\\
=&\SPE{U}{\V(\grad\ln\lambda)}^{2}+\SPE{U}{\frac{1}{2}[X,Y]}^{2}\\
=&U(\ln\lambda)^2+\frac{1}{4}\SPE{U}{[X,Y]}^{2}.
\end{align*}
Now, since the vertical part of the Lie bracket of horizontal vector fields is a tensor, the term
$\frac{1}{4}\SPE{U}{[X,Y]}^{2}$ is in fact independent of our choice of orthonormal basis
$\{X,Y\}$ for the horizontal space. To see this, suppose $a^2+b^2=1$, then
\begin{align*}
\SPE{U}{[aX+bY,bX-aY]}^{2}=&\SPE{U}{-a^2[X,Y]+b^2[Y,X]}^{2}\\
=&(-1)^{2}\SPE{U}{[X,Y]}^{2}
\end{align*}

The proof of Theorem \ref{Jon-Curv} follows from the calculation above by polarizing twice, once in $X$
and once in $U$, but we give a direct proof below.

\begin{proof}
From Proposition \ref{Gud-curv} (iii) the curvature of a horizontally conformal
submersion with totally geodesic fibers is
\[\SPE{R(U\wedge X}{Y\wedge V}=\SPE{(\nabl{U}{A})_{X}Y}{V}+\SPE{A^{*}_{X}U}{A^{*}_{Y}V}-2 V(\ln \lambda)\SPE{A_{X}Y}{U}\]
for any $U,V\in\V$ and any $X,Y\in\H$.
Both sides of the expression are tensors, so we may extend the vectors to vector fields in any way we choose.
\begin{align*}
\SPE{R(X_{\theta}\wedge U)}{X_{\theta}\wedge V}
=&\cos^{2}(\theta)\SPE{R(X\wedge U)}{X\wedge V}+\sin^{2}(\theta)\SPE{R(Y\wedge U)}{Y\wedge V}\\
&+\cos(\theta)\sin(\theta)\left(\SPE{R(X\wedge U)}{Y\wedge V}+\SPE{R(Y\wedge U)}{X\wedge V}\right).
\end{align*}
Extend $X,Y,U,V$ to unit vector fields, then $2\SPE{\nabla_{U}X}{X}=U\SPE{X}{X}=0$ and
\begin{align*}
\SPE{R(X\wedge U)}{X\wedge V}=&\SPE{(\nabla_{U}A)_{X}X}{V}+\SPE{A^{*}_{X}U}{A^{*}_{X}V}-2V(\ln\lambda)\SPE{A_{X}X}{U}\\
=&\SPE{\nabla_{U}(A_{X}X)}{V}-\SPE{A_{\nabla_{U}X}X}{V}-\SPE{A_{X}(\nabla_{U}X)}{V}\\
&+\SPE{\SPE{A^{*}_{X}U}{X}X+\SPE{A^{*}_{X}U}{Y}Y}{\SPE{A^{*}_{X}V}{X}X+\SPE{A^{*}_{X}V}{Y}Y}\\
&-2V(\ln\lambda)\SPE{\V(\grad\ln\lambda)}{U}\\
=&\SPE{\nabla_{U}\V(\grad\ln\lambda)}{V}-\SPE{A_{\nabla_{U}X}X+A_{X}(\nabla_{U}X)}{V}\\
&+\SPE{A_{X}X}{U}\SPE{A_{X}X}{V}+\SPE{A_{X}Y}{U}\SPE{A_{X}Y}{V}-2V(\ln\lambda)U(\ln\lambda)\\
=&\SPE{\nabla_{U}\V(\grad\ln\lambda)}{V}-\SPE{\SPE{\nabla_{U}X}{X}\V(\grad \ln\lambda)}{V}\\
&+U(\ln\lambda)V(\ln\lambda)+\frac{1}{4}\SPE{[X,Y]}{U}\SPE{[X,Y]}{V}-2U(\ln\lambda)V(\ln\lambda)\\
=&\SPE{\nabla_{U}\V(\grad\ln\lambda)}{V}+U(\ln\lambda)V(\ln\lambda)\\
&+\frac{1}{4}\SPE{[X,Y]}{U}\SPE{[X,Y]}{V}-2U(\ln\lambda)V(\ln\lambda).
\end{align*}
A similar calculation shows that this equals $\SPE{R(Y\wedge U)}{Y\wedge V}$.

Now since we extended to unit vector fields $\SPE{\nabla_{U}X}{Y}=-\SPE{X}{\nabla_{U}Y}$, so
\begin{align*}
\SPE{R(X\wedge U)}{Y\wedge V}+\SPE{R(Y\wedge U)}{X\wedge V}=&
\SPE{\nabla_{U}(A_{X}Y)}{V}-\SPE{A_{\nabla_{U}X}Y}{V}-\SPE{A_{X}(\nabla_{U}Y)}{V}\\
&+\SPE{\nabla_{U}(A_{Y}X)}{V}-\SPE{A_{\nabla_{U}Y}X}{V}-\SPE{A_{Y}(\nabla_{U}X)}{V}\\
&+\SPE{A_{X}X}{U}\SPE{A_{Y}X}{V}+\SPE{A_{X}Y}{U}\SPE{A_{Y}Y}{V}\\
&+\SPE{A_{Y}Y}{U}\SPE{A_{X}Y}{V}+\SPE{A_{Y}X}{U}\SPE{A_{X}X}{V}\\
&-2V(\ln\lambda)\SPE{A_{X}Y}{U}-2V(\ln\lambda)\SPE{A_{Y}X}{U}\\
=&\SPE{\nabla_{U}(A_{X}Y)}{V}+\SPE{\nabla_{U}(A_{Y}X)}{V}\\
&-\SPE{A_{\nabla_{U}X}Y}{V}-\SPE{A_{Y}(\nabla_{U}X)}{V}\\
&-\SPE{A_{X}(\nabla_{U}Y)}{V}-\SPE{A_{\nabla_{U}Y}X}{V}\\
&+\SPE{\V(\grad\ln\lambda)}{U}\SPE{A_{Y}X+A_{X}Y}{V}\\
&+\SPE{A_{X}Y+A_{Y}X}{U}\SPE{\V(\grad\ln\lambda)}{V}\\
&-2V(\ln\lambda)\SPE{A_{X}Y+A_{Y}X}{U}\\
=&\SPE{\nabla_{U}\V([X,Y])}{V}+\SPE{\nabla_{U}\V([Y,X])}{V}\\
&-\SPE{\SPE{X}{\nabla_{U}Y}\V(\grad\ln\lambda)}{V}\\
&-\SPE{\SPE{Y}{\nabla_{U}X}\V(\grad\ln\lambda)}{V}\\
=&0.
\end{align*}
Thus, the value of $\SPE{R(X_{\theta}\wedge U)}{X_{\theta}\wedge V}$ does not depend on $\theta$.
\end{proof}

\section{Implications for Einstein manifolds}

Proposition \ref{Gud-curv} (ii) says that for a horizontally conformal submersion $\phi:(M,g)\to(N,h)$ with
totally geodesic fibers the curvature operator of $M$ satisfies
\[\SPE{R(U\wedge V)}{W\wedge X}=0\]
for all $U,V,W\in\V$ and all $X\in\H$.

Let $\{U_{k}\}$ be an orthonormal basis for $\V$ and $\{X,Y\}$ be an orthonormal basis for $\H$.
If we assume that the domain $M$ is an Einstein manifold then
\begin{align*}
0=\Ric(X,U)&=\SPE{R(X\wedge Y)}{Y\wedge U}+\sum_{k}\SPE{R(X\wedge U_{k})}{U_{k}\wedge U}\\
&=\SPE{R(X\wedge Y)}{Y\wedge U}
\end{align*}
for all $U\in\V$. This means that the curvature operator $R$ splits into invariant components
\[\wedge^{2}T_{p}M=(\wedge^{2}\V\oplus\wedge^{2}\H)\oplus W,\]
where $W$ is generated by the mixed vectors, that is,
\[R(\wedge^{2}\V\oplus\wedge^{2}\H)\subseteq\wedge^{2}\V\oplus\wedge^{2}\H\textrm{ and }R(W)\subseteq W.\]
Thus, the eigenvalues of $R$ are the union of the eigenvalues
of $R|_{\wedge^{2}\V\oplus\wedge^{2}\H}$ and $R|_{W}$.

We can define a complex structure $J$ on $W$ by $J(X\wedge U)=Y\wedge U$ and $J(Y\wedge U)=-X\wedge U$.
The curvature tensor $R|_{W}$ is, due to Theorem \ref{Jon-Curv},
represented by an Hermitian matrix $H$ with respect to this complex structure.
Let $e_{j}$ be an eigenvector to the Hermitian matrix $H$, then $e_{j}$ and $J e_{j}$ represent different real
eigenvectors for $R|_{W}$ with the same eigenvalue. Thus $R|_{W}$ and therefore $R$ has at least
$\dim(M)-2$ double eigenvalues. We get
\begin{proposition}
Let $(M,g)$ be an Einstein manifold and $(N^{2},h)$ be a Riemannian surface.
Let $R$ be the curvature operator of $(M,g)$ at $p\in M$. If there is a submersive harmonic
morphism $\phi:(M,g)\to(N^{2},h)$ with totally geodesic fibers then $R$ has at least
$\dim(M)-2$ pairs of eigenvalues.
\end{proposition}
In particular, the relationship between the determinants of $R|_{W}$ and $H$ is $\det(R|_{W})=\det(H)^{2}$.
So if $F$ is the characteristic polynomial of $H$ and $f$ the characteristic polynomial of $R_{W}$,
then $f=F^{2}$ and $F$ is a factor of $\gcd(f,f^{'})$. Thus $\gcd(f,f^{'})$ is a polynomial of degree at least
$\deg(F)=\dim(M)-2$.

\section{Examples}

We give an example of a five dimensional manifold that does not have any
conformal foliations with totally geodesic fibers, not even locally. The two homogeneous
Einstein manifolds below were found by Alekseevsky in \cite{Alek}, but we use the notation of \cite{Nikon}.

\begin{example}\label{NoTot}
Let $S$ be the five dimensional homogeneous Einstein manifold given in Theorem 1(5) of \cite{Nikon}.
This is a solvable simply connected Lie group corresponding to the Lie algebra $\s$ given by an orthonormal basis
$\{A,X_{1},X_{2},X_{3},X_{4}\}$ with Lie brackets
\begin{align*}
&[X_{1},X_{2}]=\sqrt{\frac{2}{3}} X_{3},\,[X_{1},X_{3}]=\sqrt{\frac{2}{3}} X_{4},\\
&[A,X_{j}]=\frac{j}{\sqrt{30}}X_{j},\, j=1,2,3,4.
\end{align*}
A long but straightforward calculation shows that the curvature operator is given by
\[\frac{1}{30}\left[\begin{array}{cccccccccc}
13 & -2\sqrt{5} & -4\sqrt{5} & 0 & 0 & 0 & 0 & 0 & 0 & 0\\
-2\sqrt{5} & 4 & 0 & 0 & 0 & 0 & 0 & 0 & 0 & 0\\
-4\sqrt{5} & 0 & 16 & 0 & 0 & 0 & 0 & 0 & 0 & 0\\
0 & 0 & 0 & 8 & 0 & 0 & 0 & 0 & 0 & 0\\
0 & 0 & 0 & 0 & 1 & 0 & 0 & \sqrt{5} & 0 & \sqrt{5}\\
0 & 0 & 0 & 0 & 0 & 9 & -3\sqrt{5} & 0 & -3\sqrt{5} & 0\\
0 & 0 & 0 & 0 & 0 & -3\sqrt{5} & 17 & 0 & 5 & 0\\
0 & 0 & 0 & 0 & \sqrt{5} & 0 & 0 & 1 & 0 & 5\\
0 & 0 & 0 & 0 & 0 & -3\sqrt{5} & 5 & 0 & -1 & 0\\
0 & 0 & 0 & 0 & \sqrt{5} & 0 & 0 & 5 & 0 & 7
\end{array}\right]\]
with respect to the basis
\[\{X_{1}\wedge X_{3},X_{2}\wedge A,X_{4}\wedge A,X_{2}\wedge X_{4},
X_{1}\wedge A,X_{3}\wedge A,X_{1}\wedge X_{2},X_{2}\wedge X_{3},X_{1}\wedge X_{4},X_{3}\wedge X_{4}\}.\]
Let $f$ be the characteristic polynomial, then $\gcd(f(x),f'(x))=-\frac{4}{15}+x$, which is a
polynomial of degree $1<3$, so there are no conformal foliations with totally geodesic fibers.
\end{example}

One way to produce foliations on a Lie group $G$ is to find a subalgebra $\v$
of the Lie algebra $\g$ of $G$. The Riemannian metric on $G$ is the left translation of the
scalar product on $\g$. If $\v$ corresponds to a closed subgroup $K$ we foliate by left translating
this subgroup, $\F=\{L_{g}K\}_{g\in G}$.
The foliation has totally geodesic fibers if and only if
\begin{align*}
\SPE{B_{U}V}{X}&=-\frac{1}{2}(\SPE{[X,U]}{V}+\SPE{[X,V]}{U})=0
\end{align*}
for all $U,V\in\v$ and all $X\in\h=\v^{\bot}$ and is conformal if
\begin{align*}
(\L_{V}g)(X,Y)&=-\frac{1}{2}(\SPE{[V,X]}{Y}+\SPE{[V,Y]}{X})=\nu(V)\SPE{X}{Y}
\end{align*}
for all $V\in\v$ and all $X,Y\in\h$ where $\nu$ is a linear functional on $\v$.

For the Example \ref{NoTot}.
If we define a foliation by setting $\v=\{A,X_{2},X_{4}\}$ and
$\h=\{X_{1},X_{3}\}$ in the procedure above, then we get a foliation
with totally geodesic fibers but it is not conformal.
If instead we define a foliation by setting $\v=\{X_{2},X_{3},X_{4}\}$ and
$\h=\{A,X_{1}\}$, then we get a conformal foliation but
this does not have totally geodesic fibers, in fact, not even minimal fibers.

We also give an example of a five dimensional manifold with a conformal foliation with totally geodesic fibers,
and see how the curvature operator behaves.
\begin{example}
Let $S$ be the five dimensional homogeneous Einstein manifold given in Theorem 1(4) of \cite{Nikon}.
This is a solvable simply connected Lie group corresponding to the Lie algebra $\s$ given by an orthonormal basis
$\{A,X_{1},X_{2},X_{3},X_{4}\}$ with Lie brackets
\begin{align*}
&[X_{1},X_{2}]=\sqrt{\frac{2}{3}} X_{3},\\
&[A,X_{1}]=\frac{2}{\sqrt{33}}X_{1},\,[A,X_{2}]=\frac{2}{\sqrt{33}}X_{2},\,
[A,X_{3}]=\frac{4}{\sqrt{33}}X_{3},\,
[A,X_{4}]=\frac{3}{\sqrt{33}}X_{4}.
\end{align*}
If we left translate $\v=\{A,X_{3},X_{4}\}$ and $\h=\{X_{1},X_{2}\}$ we get a conformal
foliation with totally geodesic fibers.
The curvature operator is given by
\[\frac{1}{66}\left[\begin{array}{cccccccccc}
41 & -4\sqrt{22} & 0 & 0 & 0 & 0 & 0 & 0 & 0 & 0\\
-4\sqrt{22} & 32 & 0 & 0 & 0 & 0 & 0 & 0 & 0 & 0\\
0 & 0 & 18 & 0 & 0 & 0 & 0 & 0 & 0 & 0\\
0 & 0 & 0 & 24 & 0 & 0 & 0 & 0 & 0 & 0\\
0 & 0 & 0 & 0 & 8 & 0 & 0 & 2\sqrt{22} & 0 & 0\\
0 & 0 & 0 & 0 & 0 & 8 & -2\sqrt{22} & 0 & 0 & 0\\
0 & 0 & 0 & 0 & 0 & -2\sqrt{22} & 5 & 0 & 0 & 0\\
0 & 0 & 0 & 0 & 2\sqrt{22} & 0 & 0 & 5 & 0 & 0\\
0 & 0 & 0 & 0 & 0 & 0 & 0 & 0 & 12 & 0\\
0 & 0 & 0 & 0 & 0 & 0 & 0 & 0 & 0 & 12
\end{array}\right]\]
with respect to the basis
\[\{X_{1}\wedge X_{2},X_{3}\wedge A,X_{4}\wedge A,X_{3}\wedge X_{4},X_{1}\wedge A,X_{2}\wedge A,
X_{1}\wedge X_{3},X_{2}\wedge X_{3},X_{1}\wedge X_{4},X_{2}\wedge X_{4}\}.\]
We see that the curvature operator
satisfies the conclusions of Theorem \ref{Jon-Curv}.
\end{example}

We now give some details about Example 6.1 and 6.2 of \cite{Gud-Sven-2013}.
\begin{example}\label{GudSvenEx1}
For Example 6.1 we let $\g_{1}$ be the Lie algebra generated by the orthonormal vectors
$\{W,X_{1},\ldots,X_{n+1}\}$ with Lie brackets
\[[W,X_{k}]=X_{k+1}\textrm{ for }k=1,\ldots,n.\]
Let $\v=\sp\{X_{2},\ldots,X_{n+1}\}$, this is a subalgebra of codimension $2$. The Ricci curvature of $G_{1}$ the
simply connected Lie group related to $\g_{1}$
satisfies $\Ric(W,W)-\Ric(X_{1},X_{1})=\frac{1-n}{2}$.
In this case
\[\SPE{B_{X_{2}}X_{3}}{W}=\frac{1}{2},\]
and thus the foliation defined by $\v$ it is not totally geodesic.
\end{example}
\begin{example}\label{GudSvenEx2}
For Example 6.2, let $\g_{2}$ be generated by the orthonormal vectors
$\{W,X_{1},\ldots,X_{n}\}$ with Lie brackets
\[[W,X_{k}]=\alpha_{k}X_{k}\textrm{ where }\alpha_{k}\in\rn\textrm{ for }k=1,\ldots,n.\]
Let $\v=\sp\{X_{2},\ldots,X_{n}\}$, this is a subalgebra of codimension $2$. The Ricci curvature of $G_{2}$ the
simply connected Lie group related to $\g_{2}$ satisfies
\[\Ric(X_{1})=-\alpha_{1}(\alpha_{1}+\ldots+\alpha_{n})X_{1}\textrm{ and }
\Ric(W)=-(\alpha_{1}^{2}+\ldots+\alpha_{n}^{2})W.\]
The fibers of the foliation given by $\v$ are totally geodesic if and only if
$\alpha_{2}=\ldots=\alpha_{n}=0$, in which case the Ricci curvature condition is satisfied.
\end{example}

\section{Holomorphic harmonic morphisms}

In this section we will show that the Ricci curvature condition still holds under weaker conditions than
totally geodesic fibers.

\begin{definition}
Let $\phi:(M^{2m},g,J)\to (N^{2},h,J^{N})$ be a horizontally conformal map between almost
Hermitian manifolds. We say that $J$ is \textbf{adapted} to $\phi$ if $\phi$ is holomorphic with respect to $J$.
\end{definition}

If $M$ is $4$-dimensional then locally there exist exactly two adapted almost complex structures
(up to sign), in higher dimensions there are several such structures.
If $J$ is adapted then $J\V\subseteq\V$, $J\H\subseteq\H$ and $J$ commutes with the
orthogonal projections of $TM$ onto $\V$ and $\H$.

An almost complex structure $J$ is integrable if and only if the \textbf{Nijenhuis tensor},
\[N_{J}(Z,W)=[Z,W]+J[JZ,W]+J[Z,JW]-[JZ,JW],\]
is zero, in which case we say that $J$ is a complex structure.

\begin{definition}
Let $\F$ be a foliation on an almost Hermitian manifold $(M^{2m},g,J)$ with vertical distribution $\V$.
We say that the almost complex structure is compatible with the second fundamental form
$B$ of $\F$ if $J B_{U}V=B_{JU}V=B_{U}JV$ for all $U,V\in\V$.
\end{definition}

\begin{definition}
Let $\F$ be a foliation on an almost Hermitian manifold $(M,g,J)$ with vertical distribution $\V$.
$\F$ is said to have \textbf{superminimal} fibers if $\nabla_{U}J=0$ for all $U\in\V$.
\end{definition}

It is known that if a conformal foliation on an almost Hermitian manifold has superminimal fibers then
the almost complex structure is compatible with the second fundamental form, Section 7.8 in \cite{BW-book}
and is integrable, Proposition 7.9.1 of \cite{BW-book}.

\begin{lemma}
Let $(M,g,J)$ be an almost Hermitian manifold. If $J$ is compatible with the second fundamental
form $B$, then
\[B^{*}_{U}JX=-B^{*}_{JU}X=JB^{*}_{U}X,\]
for all $U\in\V$ and $X\in\H$.
\end{lemma}

\begin{proof}
The proof is a simple calculation. Let $V\in\V$, then
\begin{align*}
\SPE{B^{*}_{U}JX}{V}=&\SPE{JX}{B_{U}V}\\
=&-\SPE{X}{J B_{U}V}\\
=&-\SPE{X}{B_{JU}V}=-\SPE{B^{*}_{JU}X}{V}\\
=&-\SPE{X}{B_{U}JV}=-\SPE{B^{*}_{U}X}{JV}=\SPE{JB^{*}_{U}X}{V},
\end{align*}
since $V$ is arbitrary the lemma follows.
\end{proof}

We will show that the Ricci curvature condition holds in any even dimension if one of the adapted almost
complex structures is integrable and compatible with the second fundamental form.
The result is similar to Proposition 6.3 in \cite{LouPan} that deals with the $4$-dimensional case.

Wood showed, see Proposition 3.9 of \cite{Wood92}, 
that in four dimensions the adapted almost complex structure is integrable if and only if
the fibers of the foliation are superminimal. Thus in four dimensions we only have to assume
that the adapted almost complex structure is integrable.

\begin{theorem}
Let $\phi:M^{2m}\to N^{2}$ be a harmonic morphism between Hermitian manifolds $(M^{2m},g,J)$ and
$(N^{2},h,J^{N})$. Suppose that $J$ is adapted to $\phi$ and compatible with the second
fundamental form $B$. Then
\[\Ric(X,X)=\Ric(Y,Y)\textrm{ and }\Ric(X,Y)=0\]
for $X,Y\in\H$ orthonormal.
\end{theorem}

\begin{proof}
Let $\{X,Y\}$ be an orthonormal basis for $\H$ and $\{U_{i},V_{i}\}_{i=1}^{m}$ be an orthonormal basis for $\V$
chosen in such a way that $JX=Y$ and $JU_{i}=V_{i}$. We have
\begin{align*}
\Ric(X,X)&=\sum_{i}(R(X,U_{i},U_{i},X)+R(X,V_{i},V_{i},X))+R(X,Y,Y,X)\\
\Ric(Y,Y)&=\sum_{i}(R(Y,U_{i},U_{i},Y)+R(Y,V_{i},V_{i},Y))+R(Y,X,X,Y)\\
\end{align*}
From the symmetries of the curvature operator $R(X,Y,Y,X)=R(Y,X,X,Y)$.
The curvature is given by Proposition \ref{Gud-curv} (iii). That the terms not including $B^{*}$, that is,
\[\SPE{(\nabl{U}{A})_{X}Y}{V}+\SPE{A^{*}_{X}U}{A^{*}_{Y}V}-2 V(\ln \lambda)\SPE{A_{X}Y}{U},\]
satisfy the Ricci curvature condition is clear from Theorem \ref{Jon-Curv}. We denote the terms that contain
$B^{*}$ by $\tilde{R}$,
\[\SPE{\tilde{R}(U\wedge X}{Y\wedge V}=\SPE{(\nabla_{X}B^{*})_{U}Y}{V}-\SPE{B^{*}_{V}Y}{B^{*}_{U}X},\]
where by definition
\[\SPE{(\nabla_{X}B^{*})_{U}X}{U}=\SPE{\nabla_{X}(B^{*}_{U}X)}{U}-\SPE{B^{*}_{\nabla_{X}U}X}{U}-\SPE{B^{*}_{U}\nabla_{X}X}{U}.\]
Thus to prove $\Ric(X,X)=\Ric(Y,Y)$, we want to show that
\begin{align*}
\tilde{R}(X,U_{i},U_{i},X)+\tilde{R}(X,V_{i},V_{i},X)=\tilde{R}(Y,U_{i},U_{i},Y)+\tilde{R}(Y,V_{i},V_{i},Y)
\end{align*}
for each $i$. Since we prove it for each $i$ we will suppress the index and assume $JU=V$.

Now define $F_{1},F_{2},F_{3}$ and $F_{4}$ by
\begin{align*}
F_{1}(X)&=\SPE{\nabla_{X}(B^{*}_{U}X)}{U}+\SPE{\nabla_{X}(B^{*}_{V}X)}{V}\\
F_{2}(X)&=\SPE{B^{*}_{\nabla_{X}U}X}{U}+\SPE{B^{*}_{\nabla_{X}V}X}{V}\\
F_{3}(X)&=\SPE{B^{*}_{U}\nabla_{X}X}{U}+\SPE{B^{*}_{V}\nabla_{X}X}{V}\\
F_{4}(X)&=|B^{*}_{U}X|^{2}+|B^{*}_{V}X|^{2}.
\end{align*}
Then
\[\tilde{R}(X,U,U,X)+\tilde{R}(X,V,V,X)=F_{1}(X)-F_{2}(X)-F_{3}(X)-F_{4}(X).\]
We will show that $F_{j}(X)=F_{j}(Y)$ for $j=1,\ldots,4$, which implies $\Ric(X,X)=\Ric(Y,Y)$.

We start with $F_{1}$
\begin{align*}
F_{1}(Y)=&\SPE{\nabla_{JX}(B^{*}_{U}JX)}{U}+\SPE{\nabla_{JX}(B^{*}_{V}JX)}{V}\\
=&-\SPE{\nabla_{JX}(B^{*}_{JU}X)}{U}-\SPE{\nabla_{JX}(B^{*}_{JV}X)}{V}\\
=&-\SPE{\nabla_{B^{*}_{JU}X}JX+[JX,B^{*}_{JU}X]}{U}-\SPE{\nabla_{B^{*}_{JV}X}JX+[JX,B^{*}_{JV}X]}{V}\\
=&-\SPE{J\nabla_{B^{*}_{JU}X}X-[JX,JB^{*}_{U}X]}{U}-\SPE{J\nabla_{B^{*}_{JV}X}X-[JX,JB^{*}_{V}X]}{V}\\
=&\SPE{\nabla_{B^{*}_{JU}X}X}{JU}+\SPE{[JX,JB^{*}_{U}X]}{U}+\SPE{\nabla_{B^{*}_{JV}X}X}{JV}+\SPE{[JX,JB^{*}_{JU}X]}{V}\\
=&\SPE{\nabla_{B^{*}_{V}X}X}{V}+\SPE{\nabla_{B^{*}_{U}X}X}{U}+\SPE{[JX,JB^{*}_{U}X]-J[JX,B^{*}_{U}X]}{U}\\
=&\SPE{\nabla_{X}B^{*}_{V}X-[X,B^{*}_{V}X]}{V}+\SPE{\nabla_{X}B^{*}_{U}X-[X,B^{*}_{U}X]}{U}\\
&+\SPE{[JX,JB^{*}_{U}X]-J[JX,B^{*}_{U}X]}{U}\\
=&\SPE{\nabla_{X}(B^{*}_{U}X)}{U}+\SPE{\nabla_{X}(B^{*}_{V}X)}{V}\\
&+\SPE{[JX,JB^{*}_{U}X]-J[JX,B^{*}_{U}X]-[X,B^{*}_{U}X]-J[X,JB^{*}_{U}X]}{U}\\
=&\SPE{\nabla_{X}(B^{*}_{U}X)}{U}+\SPE{\nabla_{X}(B^{*}_{V}X)}{V}-\SPE{N_{J}(X,B^{*}_{U}X)}{U}\\
=&\SPE{\nabla_{X}(B^{*}_{U}X)}{U}+\SPE{\nabla_{X}(B^{*}_{V}X)}{V}\\
=&F_{1}(X).
\end{align*}

Next is $F_{2}$
\begin{align*}
F_{2}(Y)=&\SPE{B^{*}_{\nabla_{JX}U}JX}{U}+\SPE{B^{*}_{\nabla_{JX}V}JX}{V}\\
=&\SPE{JX}{B_{\nabla_{JX}U}U}+\SPE{JX}{B_{\nabla_{JX}V}V}\\
=&\SPE{JX}{B_{U}(\nabla_{JX}U)}+\SPE{JX}{B_{V}(\nabla_{JX}V)}\\
=&\SPE{JX}{B_{U}(\nabla_{U}JX+[JX,U])}+\SPE{JX}{B_{V}(\nabla_{V}JX+[JX,V])}\\
=&\SPE{JX}{B_{U}(\nabla_{U}JX)}+\SPE{JX}{B_{V}(\nabla_{V}JX)}\\
&+\SPE{JX}{B_{U}([JX,U])}+\SPE{JX}{B_{JU}([JX,JU])}\\
=&\SPE{X}{B_{U}(\nabla_{U}X)}+\SPE{X}{B_{V}(\nabla_{V}X)}+\SPE{X}{B_{U}([JX,JU]-J[JX,U])}\\
=&\SPE{X}{B_{U}(\nabla_{X}U-[X,U])}+\SPE{X}{B_{V}(\nabla_{X}V-[X,V])}\\
&+\SPE{X}{B_{U}([JX,JU]-J[JX,U])}\\
=&\SPE{B^{*}_{\nabla_{X}U}X}{U}+\SPE{B^{*}_{\nabla_{X}V}X}{V}\\
&+\SPE{X}{B_{U}(-[X,U]-J[X,JU]+[JX,JU]-J[JX,U])}\\
=&\SPE{B^{*}_{\nabla_{X}U}X}{U}+\SPE{B^{*}_{\nabla_{X}V}X}{V}-\SPE{X}{B_{U}(N_{J}(X,U))}\\
=&F_{2}(X).
\end{align*}

Now we show that $F_{3}(X)=0$, the same is true for $F_{3}(Y)$,
\begin{align*}
F_{3}(X)=&\SPE{B^{*}_{U}\nabla_{X}X}{U}+\SPE{B^{*}_{V}\nabla_{X}X}{V}\\
=&\SPE{\nabla_{X}X}{B_{U}U+B_{V}V}\\
=&0.
\end{align*}

The last one $F_{4}$ follows from
\begin{align*}
F_{4}(X)=&|B^{*}_{U}X|^{2}+|B^{*}_{V}X|^{2}\\
=&|B^{*}_{V}JX|^{2}+|B^{*}_{U}JX|^{2}\\
=&|B^{*}_{U}Y|^{2}+|B^{*}_{V}Y|^{2}\\
=&F_{4}(Y)
\end{align*}

We have shown that $\Ric(X,X)=\Ric(Y,Y)$ for any orthonormal basis, since
$\{\frac{1}{\sqrt{2}}(X+Y),\frac{1}{\sqrt{2}}(X-Y)\}$ also is an orthonormal basis we have
\[\Ric(X,Y)=\frac{1}{2}(\Ric(\frac{X+Y}{\sqrt{2}},\frac{X+Y}{\sqrt{2}})
-\Ric(\frac{X-Y}{\sqrt{2}},\frac{X-Y}{\sqrt{2}})=0.\]
\end{proof}

We take another look at Example \ref{GudSvenEx1}. Any adapted
almost complex structure $J$ must satisfy $JW= X_{1}$ and $J(\v)\subseteq\v$. Thus
\begin{align*}
N_{J}(W,X_{n+1})=&[W,X_{n+1}]+J[W,J X_{n+1}]+J[JW,X_{n+1}]+[J W,J X_{n+1}]\\
=&J[W,J X_{n+1}]\neq 0,
\end{align*}
and non of the adapted almost complex structures are integrable.

%%%%%%%%%%%%%%%%%%%%%%%%%%%%%%%%%%%%%%%%%%%%%%%%%%%%%%%%%%%%%%%%%%%%%%%%%%%%%%%%
%%%%%%%%%%%%%%%%%%%%%%%%%%%%%%%%%%%% Slutford %%%%%%%%%%%%%%%%%%%%%%%%%%%%%%%%%%
%%%%%%%%%%%%%%%%%%%%%%%%%%%%%%%%%%%%%%%%%%%%%%%%%%%%%%%%%%%%%%%%%%%%%%%%%%%%%%%%

\end{document}